\documentclass[11pt,leqno]{amsart}
\usepackage{verbatim,amssymb,amsfonts,latexsym,amsmath,amsthm,tikz,xspace,pgfplots,mathtools}

\newtheorem{theorem}{Theorem}[section]
\newtheorem{lemma}[theorem]{Lemma}
\newtheorem{corollary}[theorem]{Corollary}

\newtheorem{question}[theorem]{Question}

\newtheorem*{theorem*}{Theorem}
\newtheorem*{corollary*}{Corollary}

\newtheorem*{claim}{Claim}
\newtheorem*{sub-claim}{sub-claim}
\theoremstyle{definition}

\newtheorem{definition}[theorem]{Definition}

\theoremstyle{remark}

\newtheorem*{definition*}{Definition}

\newcommand{\N}{\mathbb{N}}

\newcommand{\A}{\mathcal{A}}
\newcommand{\B}{\mathcal{B}}
\newcommand{\C}{\mathcal{C}}

\newcommand{\explicitSet}[1]{\left\lbrace #1 \right\rbrace}
\newcommand{\brackets}[1]{\left\langle #1 \right\rangle}
\newcommand{\set}[2]{\explicitSet{#1 \colon #2}}
\newcommand{\seq}[2]{\brackets{#1 \colon #2}}
\newcommand{\<}{\langle}
\renewcommand{\>}{\rangle}
\renewcommand{\a}{\alpha}
\renewcommand{\b}{\beta}
\newcommand{\g}{\gamma}
\newcommand{\dlt}{\delta}

\newcommand{\z}{\zeta}
\renewcommand{\k}{\kappa}
\newcommand{\s}{\sigma}

\newcommand{\w}{\omega}
\newcommand{\0}{\emptyset}
\newcommand{\sub}{\subseteq}
\newcommand{\rest}{\!\restriction\!}

\newcommand{\card}[1]{\left\lvert #1 \right\rvert}
\newcommand{\PP}{\mathbb{P}}

\newcommand{\HH}{\mathbb{H}}
\newcommand{\forces}{\Vdash}

\newcommand{\1}{\mathbf{1}}

\newcommand{\cohen}[1]{\mathrm{Fn}(#1,2)}
\newcommand{\pwmf}{\mathcal{P}(\w)/\mathrm{fin}}

\renewcommand{\AA}{\mathbb{A}}
\newcommand{\BB}{\mathbb{B}}
\newcommand{\DD}{\mathbb{D}}

\newcommand{\bdd}{\mathfrak b}

\newcommand{\ch}{\ensuremath{\mathsf{CH}}\xspace}
\newcommand{\zfc}{\ensuremath{\mathsf{ZFC}}\xspace}

\newcommand{\ma}{\ensuremath{\mathsf{MA}}\xspace}

\newcommand{\gch}{\ensuremath{\mathsf{GCH}}\xspace}

\newcommand{\axiom}{\raisebox{.5mm}{$\bigtriangledown$}\xspace}
\newcommand{\chang}{$(\aleph_{\omega+1},\aleph_\omega) \hspace{-.5mm} \twoheadrightarrow \hspace{-.5mm} (\aleph_1,\aleph_0)$\xspace}

\newcommand{\FF}{\mathbb{F}}
\newcommand{\EE}{\mathbb{E}}
\newcommand{\rel}{\trianglelefteq}

\begin{document}

\title[The independence of $\mathsf{GCH}$ and \raisebox{.5mm}{$\bigtriangledown$}]{The independence of $\mathsf{GCH}$ and a combinatorial principle related to Banach-Mazur games}
\author{Will Brian}
\address {
Will Brian\\
Department of Mathematics and Statistics\\
9201 University City Blvd.\\
Charlotte, NC 28223}
\email{wbrian.math@gmail.com}
\urladdr{wrbrian.wordpress.com}
\author{Alan Dow}
\address {
Alan Dow\\
Department of Mathematics and Statistics\\
University of North Carolina at Charlotte\\
Charlotte, NC 28223}
\email{adow@uncc.edu}
\author{Saharon Shelah}
\address {
Saharon Shelah\\
Einstein Institute of Mathematics\\
The Hebrew University of Jerusalem\\
Jerusalem, 91904, Israel\\
and Department of Mathematics\\
Rutgers University\\
New Brunswick, NJ 08854, USA}
\email{shlhetal@math.huji.ac.il}
\urladdr{http://shelah.logic.at}


\begin{abstract}
It was proved recently that Telg\'arsky's conjecture, which concerns partial information strategies in the Banach-Mazur game, fails in models of $\mathsf{GCH}+\square$. The proof introduces a combinatorial principle that is shown to follow from $\mathsf{GCH}+\square$, namely:
\begin{itemize}
\item[\raisebox{.5mm}{$\bigtriangledown$}:] Every separative poset $\mathbb P$ with the $\kappa$-cc contains a dense sub-poset $\mathbb D$ such that $|\{ q \in \mathbb D \,:\, p \text{ extends } q \}| < \kappa$ for every $p \in \mathbb P$. 
\end{itemize}
We prove this principle is independent of $\mathsf{GCH}$ and $\mathsf{CH}$, in the sense that \raisebox{.5mm}{$\bigtriangledown$} does not imply $\mathsf{CH}$, and $\mathsf{GCH}$ does not imply \raisebox{.5mm}{$\bigtriangledown$} assuming the consistency of a huge cardinal.

We also consider the more specific question of whether \raisebox{.5mm}{$\bigtriangledown$} holds with $\mathbb P$ equal to the weight-$\aleph_\omega$ measure algebra. We prove, again assuming the consistency of a huge cardinal, that the answer to this question is independent of $\mathsf{ZFC}+\mathsf{GCH}$.
\end{abstract}

\maketitle

\section{Introduction}

Telg\'arsky's conjecture states that for each $k \in \N$, there is a topological space $X$ such that the player {\small NONEMPTY} has a winning $(k+1)$-tactic, but no winning $k$-tactic, in the Banach-Mazur game on $X$. Recently, the first two authors, along with David Milovich and Lynne Yengulalp, proved that it is consistent for this conjecture to fail \cite{BDMY}. The proof introduces the following combinatorial principle, which implies the failure of Telg\'arsky's conjecture:
\begin{itemize}
\item[\raisebox{.5mm}{$\bigtriangledown$}:] Every separative poset $\mathbb P$ with the $\kappa$-cc contains a dense sub-poset $\mathbb D$ such that $|\{ q \in \mathbb D \,:\, p \text{ extends } q \}| < \kappa$ for every $p \in \mathbb P$. 
\end{itemize}
In \cite{BDMY}, the consistency of \axiom is proved from $\gch+\square$ via the construction of what are called $\k$-sage Davies trees, which are defined in Section 2 below. The existence of arbitrarily long $\k$-sage Davies trees implies \axiom holds for $\k$-cc posets.
It is also proved in \cite{BDMY} that \axiom implies $\bdd = \aleph_1$, or more generally that \axiom implies there is no decreasing sequence of length $\w_2$ in $\pwmf$. Therefore \axiom is independent of \zfc. 

But this raises the question of the relationship between \axiom and \gch, specifically whether either of these statements implies the other.
The purpose of this paper is to answer this question in the negative by showing that $\mathsf{GCH}$ does not imply \axiom, and \axiom does not imply $\mathsf{CH}$.

In Section~\ref{sec:Cohen}, we prove that when Cohen reals are added by forcing, the existence of arbitrarily long $\k$-sage Davies trees in the ground model suffices to guarantee that \axiom holds for $\k$-cc posets in the extension. 
Thus adding Cohen reals to a model of $\gch+\square$ produces a model of $\axiom+\neg\ch$.

On the other hand, we show in Section 3 that the Chang conjecture \chang implies that \axiom fails. This is done by directly constructing a ccc poset $\PP$ (a modified product of $\aleph_\w$ Hechler forcings) and then using \chang to show it violates \axiom. As $\gch + $\chang is consistent relative to a huge cardinal \cite{EH}, this shows that \gch does not imply \axiom unless huge cardinals are inconsistent. 
We note that finding a model of $\gch+\neg\axiom$ requires large cardinals. In fact, the proof of the consistency of \axiom in \cite{BDMY} only uses $\gch+\square$-for-singulars, and the consistency of $\gch$ plus the failure of $\square$ at any singular cardinal is known to have significant large cardinal strength \cite{CF}. 

In Section 4 we consider the more specific question of whether \raisebox{.5mm}{$\bigtriangledown$} holds with $\mathbb P$ equal to the weight-$\aleph_\omega$ measure algebra. We prove that the answer to this question is also independent of $\mathsf{ZFC}+\mathsf{GCH}$.
Once again Chang's conjecture for $\aleph_\w$ comes into the proof, and so the result is established modulo the consistency of a huge cardinal.

\section{\axiom does not imply \ch}\label{sec:Cohen}

A Davies tree is a sequence $\seq{M_\a}{\a < \nu}$ of elementary submodels of some large fragment $H_\theta$ of the set-theoretic universe such that the $M_\a$ enjoy certain coherence and covering properties. (These sequences are called ``trees'' because they are usually constructed by enumerating the leaves of a tree of elementary submodels of $H_\theta$.) These structures provide a unified framework for carrying out a wide variety of constructions in infinite combinatorics. They were introduced by R. O. Davies in \cite{Davies}, and an excellent survey of their many uses can be found in Daniel and Lajos Soukup's paper \cite{Soukups}.

Also in \cite{Soukups}, the Soukups construct a countably closed version of a Davies tree called a ``sage Davies tree'' using $\gch +\square$. These structures were generalized in \cite{BDMY} by constructing $<\!\k$-closed versions of these trees for uncountable $\k$, called $\k$-sage Davies trees. 
Roughly, $\k$-sage Davies trees of length $\nu$ allow us to take an object of size $\nu$ with ``critical substructures'' of size $<\!\k$ (such as a $\nu$-sized poset with the $\k$-cc), and to approximate the large object (size $\nu$) with a sequence of smaller ones (size $\k$).
It was proved in \cite{BDMY} that $\gch+\square$ implies the existence of arbitrarily long $\k$-sage Davies trees for every regular cardinal $\k$.

In this section, we show that if we begin with a model of set theory containing arbitrarily long $\k$-sage Davies trees, then, after adding any number of Cohen reals by forcing, \axiom holds in the extension for separative $\k$-cc posets. It follows that \axiom is consistent with any permissible value of $2^{\aleph_0}$.

Given a poset $\PP$, recall that the \emph{Souslin number} of $\PP$, denoted $S(\PP)$, is the minimum value of $\k$ such that $\PP$ has no antichains of size $\k$. 
Erd\H{o}s and Tarski proved in \cite{ET} that $S(\PP)$ is a regular cardinal for every poset $\PP$.

For every poset $\PP$, let $\axiom(\PP)$ denote the statement that \axiom holds for $\PP$, i.e., that there is a dense sub-poset $\DD$ of $\PP$ with $\card{\set{d \in \DD}{p \text{ extends }d}} < S(\PP)$ for every $p \in \PP$.

In what follows, $H_\theta$ denotes the set of all sets hereditarily smaller than some very big cardinal $\theta$.
Given two sets $M$ and $N$, we write $M \prec N$ to mean that $(M,\in)$ is an elementary submodel of $(N,\in)$.
A set $M$ is called \emph{$<\!\k$-closed} if $M^{<\k} \sub M$. If $M$ satisfies (enough of) \zfc, this is equivalent to the property $[M]^{<\k} \sub M$. 

\begin{definition}\label{def:tree}
Let $\k,\nu$ be infinite cardinals and let $p$ be some set. A \emph{$\k$-sage Davies tree for $\nu$ over $p$} is a sequence $\seq{M_\a}{\a < \nu}$ of elementary submodels of $(H_\theta,\in)$, for some ``big enough'' regular cardinal $\theta$, such that
\begin{enumerate}
\item $p \in M_\a$, $M_\a$ is $<\!\k$-closed, and $\card{M_\a} = \k$ for all $\a < \nu$.
\item $[\nu]^{<\k} \sub \bigcup_{\a < \nu}M_\a$.
\item For each $\a < \nu$, there is a set $\mathcal N_\a$ of elementary submodels of $H_\theta$ such that $\card{\mathcal N_\a} < \k$, each $N \in \mathcal N_\a$ is $<\!\k$-closed and contains $p$, and
$$\textstyle \bigcup_{\xi < \a}M_\xi \,=\, \bigcup \mathcal N_\a.$$
\item $\seq{M_\xi}{\xi < \a} \in M_\a$ for each $\a < \nu$.
\item $\bigcup_{\a < \nu} M_\a$ is a $<\!\k$-closed elementary submodel of $H_\theta$.
\end{enumerate}
\end{definition}

The following fact is proved in \cite[Theorem 3.20]{BDMY}:

\begin{theorem}\label{thm:Davies}
Assume $\gch+\square$. Let $\k,\nu$ be infinite regular cardinals with $\k < \nu$. For any set $p$, there is a $\k$-sage Davies tree for $\nu$ over $p$.
\end{theorem}

In fact, the proof in \cite{BDMY} uses a weak version of $\square$ related to the Very Weak Square principle articulated by Foreman and Magidor in \cite{Foreman&Magidor}. The following fact, which we will use below, is Lemma 3.7 in \cite{BDMY}.

\begin{lemma}
Let $\k,\nu$ be regular cardinals with $\k < \nu$, let $p$ be any set, and let $\seq{M_\a}{\a < \nu}$ be a $\k$-sage Davies tree for $\nu$ over $p$. If $\a < \b < \nu$, then
$$\a \in M_\b \quad \Leftrightarrow \quad M_\a \in M_\b \quad \Leftrightarrow \quad M_\a \sub M_\b.$$
\end{lemma}

In addition to the five properties listed above that define a $\k$-sage Davies tree, it will be convenient here to have trees with one additional property:
\begin{enumerate}
\item[(6)] For every $\a < \nu$, there is a well ordering $\sqsubset_\a$ of $M_\a$ with order type $\k$ such that if $\a < \b < \mu$ and $\a \in M_\b$, then $\sqsubset_\a\, \in M_\b.$
\end{enumerate}
It turns out that this property of $\k$-sage Davies trees is already a consequence of properties $(1)$ through $(5)$.

\begin{lemma}\label{lem:property6}
Let $\k,\nu$ be regular cardinals with $\k < \nu$ and let $p$ be some set. Every $\k$-sage Davies tree for $\nu$ over $p$ satisfies property $(6)$.
\end{lemma}
\begin{proof}
First observe that if $\a < \nu$ then $M_\a \in M_{\a+1}$. This is because $\seq{M_\xi}{\xi < \a+1} \in M_{\a+1}$ by definition, and this implies $M_\a \in M_{\a+1}$ because $M_\a$ is definable from $\seq{M_\xi}{\xi < \a + 1}$.

Because $|M_\a| = \k$, there is (in $H_\theta$) a well ordering of $M_\a$ with order type $\k$. By elementarity, there is some such well ordering of $M_\a$ in $M_{\a+1}$. For each $\a < \mu$, fix a well ordering $\sqsubset_\a$ of $M_\a$ with order type $\k$ such that $\sqsubset_\a\, \in M_{\a+1}$.
If $\a < \b < \nu$ and $\a \in M_\b$, then $\a+1 \in M_\b$ and therefore $M_{\a+1} \sub M_\b$ by the previous lemma. In particular, $\sqsubset_\a\, \in M_\b$.
\end{proof}

It will be convenient to work with complete Boolean algebras rather than arbitrary posets when proving \axiom holds in Cohen extensions. This restriction is justified by the following lemma.

\begin{lemma}
\axiom holds if and only if it holds for every poset of the form $\PP = \BB \setminus \{\mathbf{0}\}$, where $\BB$ is a complete Boolean algebra.
\end{lemma}
\begin{proof}
This is proved in \cite[Lemma 2.10]{BDMY}. Roughly, the ``only if'' direction is obvious because posets of the form $\BB \setminus \{\mathbf{0}\}$ are always separative, and the ``if'' direction is proved by showing that if $\PP$ is separative, then $\axiom(\PP)$ is equivalent to $\axiom(\text{the Boolean completion of }\PP)$. 
\end{proof}

Given a complete Boolean algebra $\BB$, $S(\BB)$ denotes the Souslin number of the poset $\BB \setminus \{\mathbf{0}\}$. Given $J \sub \BB$, $\bigwedge J$ denotes the infimum of $J$ in $\BB$ and $\bigvee J$ denotes the supremum of $J$ in $\BB$.

\begin{lemma}\label{lem:chaincondition}
Let $\BB$ be a complete Boolean algebra and let $J \sub \BB$. Then there is some $J' \sub J$ with $\card{J'} < S(\BB)$ such that $\bigwedge J' = \bigwedge J$ and $\bigvee J' = \bigvee J$.
\end{lemma}
\begin{proof}
If we delete the ``and $\bigvee J' = \bigvee J$'' from the end of the lemma, then it becomes a special case of \cite[Lemma 3.2]{BDMY}. If we delete the ``$\bigwedge J' = \bigwedge J$ and'' instead, then it follows from the previous sentence via de Morgan's laws.
Thus given $J \sub \BB$, there is some $J'_\wedge \sub J$ with $\card{J'_\wedge} < S(\BB)$ such that $\bigwedge J'_\wedge = \bigwedge J$, and there is some $J'_\vee \sub J$ with $\card{J'_\vee} < S(\BB)$ such that $\bigvee J'_\vee = \bigvee J$. Then $J' = J'_\wedge \cup J'_\vee$ satisfies the conclusion of the lemma.
\end{proof}

\begin{lemma}\label{lem:stabilize}
Let $\BB$ be a complete Boolean algebra and let $X \sub \BB$ with ${\card{X} = S(\BB})$. Then there is some $Y \sub X$ with $\card{X \setminus Y} < S(\BB)$ such that $\bigwedge Y = \bigwedge (Y \setminus Z)$ for every $Z \sub Y$ with $\card{Z} < S(\BB)$.
\end{lemma}
\begin{proof}
Let $\k = S(\BB)$. Fix $X \sub \BB \setminus \{\mathbf{0}\}$ with $\card{X} = \k$, and let $\set{b_\a}{\a < \k}$ be an enumeration of $X$ with order type $\k$. 
Let $c_\a \,=\, \bigwedge \set{b_\xi}{\xi \geq \a}$ for each $\a < \k$, and note that $\a \leq \a'$ implies $c_\a \leq c_{\a'}$.
By Lemma~\ref{lem:chaincondition}, there is some $\b < \k$ such that $\bigvee \set{c_\a}{\a < \k} \,=\, \bigvee \set{c_\a}{\a < \b}$. 
(This uses the fact that $\k$ is regular: as mentioned above, the Souslin number of a poset is always a regular cardinal.) 
Because the $c_\a$ form a non-decreasing sequence in $\BB$, this means $c_\a = c_\b$ for all $\a \geq \b$. 
Let $Y = \set{b_\xi}{\xi \geq \b}$.
If $Z \sub Y$ with $\card{Z} < \k$, then there is some $\a$ with $\b \leq \a < \k$ such that $Z \sub \set{b_\xi}{\xi < \a}$.
But then
$$\textstyle c_\b = \bigwedge Y \,\leq\, \bigwedge \left( Y \setminus Z \right) \,\leq\, \bigwedge \set{b_\xi}{\xi \geq \a} = c_\a = c_\b.$$
Therefore $\bigwedge \left( Y \setminus Z \right) = c_\b$ for any $Z \sub Y$ with $\card{Z} < \k$.
\end{proof}

If $\FF$ is a forcing poset and $A$ is a set, recall that a \emph{nice name} for a subset of $A$ is a subset $\dot X$ of $A \times \FF$ such that for each $a \in A$, $\set{p \in \FF}{(a,p) \in \dot X}$ is an antichain in $\FF$.
Given $B \sub A$, $\dot X \rest B = \dot X \cap (B \times \FF)$.
We adopt the convention of deleting a dot to denote the evaluation of a name. For example, if $\dot X$ is a nice $\FF$-name for a subset of $\mu$, then we write $\mathbf{1}_\FF \forces$ ``$X \sub \mu$.''

\begin{lemma}\label{lem:stabilize2}
Let $\FF$ be a ccc notion of forcing, let $\dot \rel$ be an $\FF$-name for a relation on some infinite cardinal $\mu$, and suppose that $\1_\FF \forces$ ``$\,(\mu,\rel)$ is a complete Boolean algebra with $S(\mu,\rel) = \k$."
Let $p \in \FF$ and let $\dot X$ be a nice name for a subset of $\mu$. 
If $p \forces$ ``$\,|X| = \k$'' then there is some $\dot Y \sub \dot X$ with $|\dot Y \setminus \dot X| < \k$ such that $p \forces$ ``$\bigwedge Y = \bigwedge (Y \setminus Z)$ for any $Z \sub \mu$ with $\card{Z} < \k$."
\end{lemma}
\begin{proof}
As $\mu$ is infinite, $\k$ must be a regular uncountable cardinal. 
Because $\FF$ has the ccc, we know that for every $\FF$-name $\dot W$ for a subset of $\mu$, if $q \in \FF$ and $q \forces$ ``$|W| < \k$", then there is some $A \sub \mu$ (in the ground model) such that $|A| < \k$ and $q \forces$ ``$W \sub A$.'' 

By Lemma~\ref{lem:stabilize}, and the existential completeness lemma, there is a name $\dot Y_0$ for a subset of $\mu$ such that $p \forces$ ``$Y_0 \sub X$ and $\card{X \setminus Y_0} < \k$ and $\bigwedge Y_0 = \bigwedge (Y_0 \setminus Z)$ for every $Z \sub Y$ with $\card{Z} < \k$." By the previous paragraph, there is some $A \sub \mu$ (in the ground model) such that $\card{A} < \k$ and $p \forces$ ``$X \setminus Y_0 \sub A$." Furthermore, $p \forces$ ``$\bigwedge ((X \setminus A) \setminus Z) = \bigwedge X \setminus (A \cup Z) = \bigwedge Y_0 = \bigwedge Y_0 \setminus A = \bigwedge X \setminus A$ for any $Z \sub \mu$ with $\card{Z} < \k$."

Let 
$\dot Y = \dot X \rest (\mu \setminus A)$.
Clearly $\dot Y \sub \dot X$ and $p \forces$ ``$Y = X \setminus A$." Because $\dot X$ is a nice name and $\FF$ has the ccc, $\set{q \in \FF}{(q,a) \in \dot X}$ is countable for every $a \in A$; therefore $|\dot X \setminus \dot Y| \leq \aleph_0 \cdot |A| < \k$. Finally, because $p \forces$ ``$Y = X \setminus A$", the last assertion of the lemma follows from the last sentence of the previous paragraph.
\end{proof}


Given a cardinal $\lambda$, let $\cohen{\lambda}$ denote the poset of finite partial functions $\lambda \to \{0,1\}$, the standard forcing poset for adding $\lambda$ Cohen reals.

\begin{theorem}\label{thm:cohen}
Suppose $V$ is a model of $\gch+\square$ (or, more generally, suppose $V$ is a model satisfying the conclusion of Theorem~\ref{thm:Davies}). If $\lambda$ is any cardinal and $G$ is $\cohen{\lambda}$-generic over $V$, then $V[G] \models \axiom$.
\end{theorem}

\begin{proof}
Let $\mu,\k$ be infinite cardinals, and let $\dot \rel$ be a $\cohen{\lambda}$-name such that $\0 \forces$ ``$(\mu,\rel)$ is a complete Boolean algebra with $S(\mu,\rel) = \k$." 
Note that this implies $\k$ is regular and uncountable.
Let $\nu$ be a regular uncountable cardinal with $\lambda,\mu \leq \nu$ and with $\k < \nu$.
Without loss of generality, we may and do assume that $0$ (the ordinal) is equal to $\mathbf{0}$ (the $\rel$-least element of $\mu$).
More precisely, we assume $\0 \forces$ ``$\mathbf{0}_{(\mu,\rel)} = 0$."

We work momentarily in the ground model.
Applying Theorem~\ref{thm:Davies}, 
let $\seq{M_\a}{\a < \nu}$ be a $\k$-sage Davies tree for $\nu$ over $(\mu,\dot \rel)$.
Applying Lemma~\ref{lem:property6}, fix for each $\a < \nu$ some well ordering $\sqsubset_\a$ of $M_\a$ with order type $\k$ such that if $\a < \b < \nu$ and $\a \in M_\b$, then $\sqsubset_\a\, \in M_\b.$

For each $x \in \bigcup_{\a < \nu}M_\a$, the \emph{level} of $x$, denoted $\mathrm{Lev}(x)$, is defined as the least $\a < \nu$ such that $x \in M_\a$.
Let $\sqsubset$ denote the well-order of $\bigcup_{\a < \nu}M_\a$ defined as follows:
\begin{itemize}
\item[$\circ$] if $\mathrm{Lev}(x) < \mathrm{Lev}(y)$, then $x \sqsubset y$.
\item[$\circ$] if $\mathrm{Lev}(x) = \mathrm{Lev}(y) = \a$, then $x \sqsubset y$ if and only if $x \sqsubset_\a y$.
\end{itemize}
We write $x \sqsubseteq y$ to mean that either $x \sqsubset y$ or $x = y$.

We now define, via recursion, a sequence $\seq{d_\g}{\g < \mu}$ of members of $\mu$. 
Simultaneously, we also define a sequence $\seq{I_\g}{\g < \mu}$ of $<\!\k$-sized subsets of $\mu$, and a sequence $\seq{\dot J_\g}{\g < \mu}$ of nice names.
These definitions take place in the extension $V[G]$, and we do not claim that any of these sequences is a member of the ground model $V$.
For the base case, let $d_0 = 0$ and let $I_0 = \dot J_0 = \0$.
For the recursive step, fix $\g < \mu$ and suppose that $d_\b$, $I_\b$, and $\dot J_\b$ are already defined for each $\b \sqsubset \g$.
If there is some $\b \sqsubset \g$ such that $0 \neq d_\b \rel \g$, then set $d_\g = 0$ and set $I_\g = \dot J_\g = \0$.
If there is no such $\b$, then let $I_\g$ denote the $\sqsubseteq$-minimal set in the ground model $V$ with the following two properties:
\begin{itemize}
\item[$\circ$] $I_\g$ is a $<\!\k$-sized subset of $\mu$.
\item[$\circ$] In $V[G]$, there is some $J \sub I_\g$ such that $0 \neq \bigwedge J \rel \g$.
\end{itemize}
Note that $I_\g$ is well-defined because $\{\g\} \in V$ and $\{\g\}$ has both these properties. (Note that this implies $I_\g \sqsubseteq \{\g\}$.)
Because of the second property of $I_\g$ listed above, there is a nice name $\dot J$ in the ground model $V$ for a subset of $I_\g$ such that, for some $p \in G$, we have $p \forces$ ``$(\dot J)_G = J \sub I_\g$ and $0 \neq \bigwedge J \rel \g$."
Let $\dot J_\g$ denote the $\sqsubseteq$-minimal nice $\cohen{\lambda}$-name in $V$ with this property.
Finally, let $d_\g = \bigwedge (\dot J_\g)_G$.

(Note: Because the $I_\g$'s and the $\dot J_\g$'s are defined in the extension, we have in the ground model a name $\dot I_\g$ and a name $\ddot J_\g$ for a nice name for a subset of $\dot I_\g$ that is forced (by $\0$) to be the $\g^{\mathrm{th}}$ element of the sequence constructed above. In particular, $p \forces$ ``$(\ddot J_\g)_G = \dot J_\g \sub I_\g = (\dot I_\g)_G$" for some $p \in G$. Recall our convention of deleting a dot to denote the evaluation of a name!)

Let $\DD = \set{d_\g}{\g < \mu \text{ and } d_\g \neq 0}$. We claim that this set $\DD$ is a witness to the fact that $\axiom(\mu,\rel)$ holds in $V[G]$.

To see that $\DD$ is a dense subset of $(\mu,\rel)$, fix some nonzero $\g < \mu$. 
If $d_\g \neq 0$, then $d_\g \in \DD$ and $d_\g \rel \g$. 
If $d_\g = 0$, then this means there is some $\b \sqsubset \g$ such that $0 \neq d_\b \rel \g$, and so $d_\b \in \DD$ and $d_\b \rel \g$.
Either way, some member of $\DD$ is $\rel \g$. As $\g$ was arbitrary, $\DD$ is dense.

For the more difficult part of the proof, we must show that every $\dlt \in \mu \setminus \{0\}$ has the property that $\card{\set{d \in \DD}{\dlt \rel d}} < \k$. Aiming for a contradiction, let us suppose otherwise. 
Fix some $\dlt \in \mu \setminus \{0\}$ such that $\card{\set{d \in \DD}{\dlt \rel d}} \geq \k$. Let $S = \set{\g < \mu}{d_\g \in \DD \text{ and } \dlt \rel d_\g}$.

Observe that $\b \neq \g$ implies $d_\b \neq d_\g$ whenever $d_\b,d_\g \in \DD$. (This is because if $\b \sqsubset \g$, then $d_\g \neq 0$ implies $d_\b \not\rel \g$ while $d_\g \rel \g$.) Therefore the map $\g \mapsto d_\g$ is injective on $S$, and we may think of $S$ simply as an indexing set for $\set{d \in \DD}{\dlt \rel d} = \set{d_\g}{\g \in S}$.

\begin{claim}
There is some $I \sub \mu$ such that $I_\g = I$ for $\geq\!\k$-many $\g \in S$.
\end{claim}
\begin{proof}[Proof of claim]
Aiming for a contradiction, let us assume the claim is false. Let $\z$ denote the least ordinal $<\!\nu$ with the property that $\mathrm{Lev}(I_\g) < \z$ for $\geq\!\k$-many $\g \in S$. 
Some such $\z$ must exist because $\card{S} \geq \k$ and $\nu$ is a regular cardinal with $\k < \nu$.

By part $(3)$ of our definition of a $\k$-sage Davies tree, there is a collection $\mathcal N$ of $<\!\k$-closed elementary submodels of $H_\theta$ such that $\card{\mathcal N} < \k$ and ${\bigcup \mathcal N = \bigcup_{\xi < \z}M_\xi}$. By our choice of $\z$ and the regularity of $\k$, some $N \in \mathcal N$ has the property that $I_\g \in N$ for $\geq\!\k$-many $\g \in S$. Fix some such $N$, let $S_N = \set{\g \in S}{I_\g \in N}$, and let $\DD_N = \set{d_\g}{\g \in S_N}$. Note that $\bigwedge \DD_N \neq 0$ because $\dlt \rel \bigwedge \DD \rel \bigwedge \DD_N$.

Applying Lemma~\ref{lem:chaincondition}, there is some $T \sub S_N$ with $|T| < \k$ such that ${\bigwedge \DD_N = \bigwedge \set{d_\g}{\g \in T}}$. 
Let $I_0 = \bigcup \set{I_\g}{\g \in T}$. Then $I_0$ is a subset of $N \cap \mu$ in $V[G]$, and $\card{I_0} < \k$. Because $\cohen{\lambda}$ has the ccc, there is a subset $I$ of $N \cap \mu$ in $V$ with $I_0 \sub I$ and $\card{I} \leq \card{I_0} \cdot \aleph_0 < \k$. 
Because $N$ is $<\!\k$-closed in $V$, we have $I \in N$.

For each $\g \in T$, there is a subset $J_\g = (\dot J_\g)_G$ of $I_\g$ with $\bigwedge J_\g = d_\g$. Note that $\bigwedge \set{d_\g}{\g \in T} = \bigwedge_{\g \in T} \bigwedge J_\g = \bigwedge \!\left( \bigcup_{\g \in T}J_\g \right)$, and let ${J = \bigcup_{\g \in T}J_\g}$. 
Now $J \sub I$, and $\bigwedge J = \bigwedge \set{\bigwedge J_\g}{\g \in T}= \bigwedge \set{d_\g}{\g \in T} = \bigwedge \DD_N$. Furthermore, $0 \neq \bigwedge \DD_N \rel d_\g$ for each $\g \in S_N$. Thus, for each $\g \in S_N$, there is a subset $J$ of $I$ such that $0 \neq \bigwedge J \rel d_\g \rel \g$.

This shows that $I$ satisfies the conditions in the definition of $I_\g$ whenever $\g \in S_N$. It follows that $I_\g \sqsubseteq I$ for all $\g \in S_N$.
Now, our definition of $\sqsubseteq$ entails that $I$ has $<\!\k$-many $\sqsubseteq$-predecessors in $\mathrm{Lev}(I)$, and each predecessor $I' \sqsubseteq I$ has $<\!\k$-many $\g \in S_N$ with $I_\g = I'$ (by our assumption at the beginning of the proof of this claim).
Therefore $\mathrm{Lev}(I_\g) = \mathrm{Lev}(I)$ for only $<\!\k$-many $\g \in S_N$. 
As $\mathrm{Lev}(I_\g) \leq \mathrm{Lev}(I)$ for all $\g \in S_N$ and $\card{S_N} \geq \k$, it follows that $\mathrm{Lev}(I_\g) < \mathrm{Lev}(I)$ for $\geq\!\k$-many $\g \in S_N$.
But $\mathrm{Lev}(I) < \z$, because $I \in N \sub \bigcup_{\xi < \z}M_\xi$, so this contradicts our choice of $\z$.
\end{proof}

Fix some $I \sub \mu$ with $\card{I} < \k$ that satisfies the conclusion of the above claim. By replacing $S$ with a size-$\k$ subset of $\set{\g \in S}{I_\g = I}$ if necessary, we may (and do) assume that $|S| = \k$, $I_\g = I$ for all $\g \in S$, and $\dlt \rel d_\g$ for all $\g \in S$.

Let $\z$ denote the least ordinal $<\!\nu$ such that there are $\k$-many $\g \in S$ with $\mathrm{Lev}(\dot J_\g) < \z$. 
(Some such $\z$ must exist because $\nu$ is a regular cardinal with $|S| = \k < \nu$.)
By replacing $S$ with $\set{\g \in S}{\mathrm{Lev}(\dot J_\g) < \z}$ if necessary, we may (and do) assume that $\mathrm{Lev}(\dot J_\g) < \z$ for all $\g \in S$.

Recall that the sequence $\seq{ \dot J_\g }{ \g < \mu }$ was defined in the extension, not in the ground model. In the ground model, we have a sequence $\seq{ \ddot J_\g }{ \g < \mu }$ of names for nice names, representing the sequence $\seq{ \dot J_\g }{ \g < \mu }$ constructed in the extension, meaning that $\0 \forces$ ``$(\ddot J_\g)_G = \dot J_\g$ for each $\g < \mu$.''

We now work in the ground model $V$. Let $\dot S$ be a nice $\cohen{\lambda}$-name for $S$, and fix some $p \in \cohen{\lambda}$ such that
\begin{align*}
p \forces \quad & |S| = \k, \\
& I_\g = I \text{ for all } \g \in S, \\
& \dlt \rel d_\g \text{ for all } \g \in S, \\
& \mathrm{Lev}((\ddot J_\g)_G) < \z \text{ for all } \g \in S, \text{ and } \\
& \text{if } \z' < \z \text{ then } \card{\set{\g \in S}{\mathrm{Lev}(\dot J_\g) < \z'}} < \k.
\end{align*}
Let $q$ be an arbitrary extension of $p$ in $\cohen{\lambda}$.

\begin{claim}
There is a nice name $\dot S' = \set{(\g_\a,q_\a)}{\a < \k} \sub \dot S$, a condition $r \supseteq q$, and a sequence $\seq{\dot K_{\g_\a}}{\a < \k}$ (in the ground model $V$) of nice names for subsets of $I$, such that $\mathrm{dom}(q_\a) \cap \mathrm{dom}(q_\b) = \0$ for all $\a \neq \b$ in $\k$, and 
$$r \forces \quad |S'| = \k \ \text{ and } \ \dot J_\g = (\ddot J_\g)_G = \dot K_\g \text{ for all } \g \in S'.$$
Furthermore, if $\dot T$ is any size-$\k$ subset of $\dot S'$ and $t \supseteq r$, then the above statement remains true when $\dot S'$ is replace by $\dot T$ and $r$ is replaced by $t$.
\end{claim}
\begin{proof}[Proof of claim]
Because $\dot S$ is a nice name for a subset of $\mu$ and $\cohen{\lambda}$ has the ccc, we may write $\dot S = \set{(\g_\a,p_\a)}{\a < \k}$, where $\g_\a < \mu$ and $p_\a \in \cohen{\lambda}$ for all $\a$, and where any particular ordinal appears only countably many times among the $\g_\a$, i.e., $\card{\set{\a < \k}{\g_\a = \g}} \leq \aleph_0$ for every $\g < \mu$. 

Letting $\dot S_1 = \dot S \setminus \set{(\g_\a,p_\a)}{p_\a \perp q}$, it is clear that $q \forces S_1 = S$. Note that $|\dot S_1| = \k$, because $q \forces$ ``$S_1 = S$ and $|S| = \k$." 

For every $(\g_\a,p_\a) \in \dot S_1$, $p_\a$ is compatible with $q$ and $q \cup p_\a \forces$ ``$(\ddot J_{\g_\a})_G$ is a nice name (in $V$) for a subset of $I$ and $\mathrm{Lev}((\ddot J_{\g_\a})_G) < \z$.'' For each such $\a$, we may therefore choose some $q_\a^0 \supseteq q \cup p_\a$ that decides $\ddot J_{\g_\a}$; that is, we choose some $q_\a^0 \supseteq q \cup p_\a$ and some nice name $\dot K_{\g_\a} \in V$ with $\mathrm{Lev}(\dot K_{\g_\a}) < \z$ such that $q_\a^0 \forces$ ``$\dot J_{\g_\a} = (\ddot J_{\g_\a})_G = \dot K_{\g_\a}$.''

By the $\Delta$-system lemma, there is some $D \sub \set{\a}{(\g_a,p_\a) \in \dot S_1}$ with $\card{D} = \k$ such that $\set{\mathrm{dom}(q_\a^0)}{\a \in D}$ is a $\Delta$-system with root $R$. 
(We allow for the possibility that this is a ``degenerate'' $\Delta$-system with $\mathrm{dom}(q_\a^0) = R$ for all $\a < \k$.)
By the pigeonhole principle, there is some $r: R \to 2$ and some $E \sub D$ with $\card{E} = \k$ such that $q_\a^0 \rest R = r$ for all $\a \in E$.
Note that $r \supseteq q \supseteq p$, because $q_\a^0 \supseteq q$ for each $\a$.
Let $\dot S_2 = \set{(\g_\a,q_\a^0)}{\a \in E}$. By relabelling and re-indexing the members of $\dot S_2$, we may write $\dot S_2 = \set{(\g_\a,q_\a^0)}{\a < \k}$.
Finally, let $q_\a = q_\a^0 \setminus r$ for all $\a$ and let $\dot S' = \set{(\g_\a,q_\a)}{\a < \k}$. 
It is clear that $r \forces$ ``$\dot S' = \dot S_2$", and this implies $r \forces$ ``$\dot J_\g = (\ddot J_\g)_G = \dot K_\g$ for all $\g \in S'$." Clearly $\mathrm{dom}(q_\a) \cap \mathrm{dom}(q_\b) = \0$ for all $\a \neq \b$ in $\k$.

Finally, suppose $\dot T \sub \dot S'$ with $|\dot T| = \k$, and fix $t \in \cohen{\lambda}$ with $t \supseteq r$.
That $t \forces |T| = \k$ follows from the fact that the domains of the $q_\a$'s are pairwise disjoint (so that any generic filter must include $\k$ of the $q_\a$'s), together with the fact that any particular ordinal appears only countably many times among the $\g_\a$. We have $t \forces$ ``$\dot J_\g = (\ddot J_\g)_G = \dot K_\g$ for all $\g \in T\,$" because $r \forces$ ``$\dot J_\g = (\ddot J_\g)_G = \dot K_\g$ for all $\g \in S'$" and $t \forces$ ``$T \sub S'$."
\end{proof}

Fix some nice name $\dot S'$ as in the claim above.

By part $(3)$ of our definition of a sage Davies tree, there is a collection $\mathcal N$ of $<\!\k$-closed closed elementary submodels of $H_\theta$ with $\card{\mathcal N} < \k$ such that $\bigcup \mathcal N = \bigcup_{\xi < \z}M_\xi$. 
By the pigeonhole principle, some $N \in \mathcal N$ has the property that $\dot K_{\g_\a} \in N$ for $\k$-many $\a < \k$.
Fix some such $N$.


Let $\dot S_N' = \set{(\g_\a,q_\a) \in \dot S'}{\dot K_{\g_\a} \in N}$.
Applying Lemma~\ref{lem:stabilize2}, there is some $\dot T \sub \dot S_N'$ with $\card{\dot S'_N \setminus \dot T} < \k$ such that 
\begin{align*}
r \forces \quad & \textstyle \bigwedge \set{\bigwedge K_\g}{\g \in T} = \bigwedge \set{\bigwedge K_\g}{\g \in T \setminus Z} \\
& \text{ for any } Z \sub \mu \text{ with } |Z| < \k.
\end{align*}
By re-labelling and re-indexing the $q_\a$ and $\g_\a$ one final time, let us write $\dot T = \set{(q_\a,\g_\a)}{\a < \k}$.

\begin{claim}
For any $\a < \k$ and any $s$ compatible with $q_\a$, if $s \forces$ ``$\,i \in K_{\g_\a}\!$'' then $q_\a \cup s \forces$ ``for any $j \in I$ with $j \not\rel i$, there is some $i' \in I$ such that $j \not\rel i'$ and $i' \in K_\g$ for $\k$-many $\g \in T$."
\end{claim}
\begin{proof}[Proof of claim]
For the proof of this claim, it is more convenient to work in a generic extension. Suppose $s$ is compatible with $q_\a$ and $s \forces$ ``$i \in K_{\g_\a}$", and let $V[H]$ be an arbitrary $\cohen{\lambda}$-generic extension of $V$ with $q_\a \cup s \in H$.

Fix $j \in I$ with $j \not\rel i$.
Because $q_\a \in H$, we have $\g_\a \in T$.
Therefore $\bigwedge \set{\bigwedge K_\g}{\g \in T} \rel \bigwedge K_{\g_\a} \rel i$. 
As $j \not\rel i$, we have $j \not\rel \bigwedge \set{\bigwedge K_\g}{\g \in T}$.
By our choice of $T$, we also have $j \not\rel \bigwedge \set{\bigwedge K_\g}{\g \in T \setminus Z}$ for any $<\!\k$-sized $Z \sub T$.
This implies there are $\k$-many $\g \in T$ such that $j \not\rel \bigwedge K_\g$.
For each such $\g$, there is some $i' \in I$ such that $i' \in K_\g$ and $j \not\rel i'$.
By the pigeonhole principle, using the fact that $\card{I} < \k$, there is some particular $i' \in I$ with $j \not\rel i'$ such that $i' \in K_\g$ for $\k$-many $\g \in T$.

Thus any generic extension $V[H]$ with $q_\a \cup s \in H$ satisfies ``for any $j \in I$ with $j \not\rel i$, there is some $i' \in I$ such that $j \not\rel i'$ and $i' \in K_\g$ for $\k$-many $\g \in T$." The claim follows.
\end{proof}

Given $i \in I$ and $\a < \k$, we write ``$i \in \mathrm{supp}(\dot K_{\g_\a})$" to mean $(i,s) \in \dot K_{\g_\a}$ for some $s \in \cohen{\lambda}$. Let 
$$I_\k = \set{i \in I}{i \in \mathrm{supp}(\dot K_{\g_\a}) \text{ for } \k \text{-many values of } \a}.$$
Note that $I_\k \sub N$ (because each $\dot K_{\g_\a}$ is in $N$). Let
$$\dot K = \set{(i,s) \in N}{i \in I_\k \text{ and } s \forces \text{``}i \in K_\g \text{ for infinitely many } \g \in T\, \text{''}}.$$
Notice that $\dot K \sub N$, although we cannot claim $\dot K \in N$.
The following claim gives us the next best thing to having $\dot K \in N$.

\begin{claim}
There is a nice name $\dot J$ for a subset of $I$, with $\dot J \in N$, such that $\0 \forces$ ``$J = K$."
\end{claim}
\begin{proof}[Proof of Claim]
For each $i \in \mathrm{supp}(\dot K)$, fix an antichain $\mathcal A_i$ in $\cohen{\lambda} \cap N$ such that $s \forces$ ``$i \in K$'' for every $s \in \mathcal A_i$, and $\mathcal A_i$ is maximal with respect to this property (i.e., if $t \in \cohen{\lambda} \cap N$ and $t \forces$ ``$i \in K$", then $t$ is compatible with some member of $\mathcal A_i$). Let $\dot J = \set{(i,s)}{i \in \mathrm{supp}(\dot K) \text{ and }s \in \mathcal A_i}$.

Clearly $\dot J$ is a nice name for a subset of $I$. Note that $i \in \mathrm{supp}(\dot K)$ implies $i \in N$. So if $(i,s) \in \dot J$, then $i,s \in N$, which implies $(i,s) \in N$. Thus $\dot J \sub N$. Also $|I| < \k$ and $\card{\A_i} = \aleph_0 < \k$ for each $i$, which implies $|\dot J| < \k$. Because $N$ is $<\!\k$-closed, $\dot J \in N$.

It is clear from our construction that $\0 \forces$ ``$J \sub K$." For the other direction, suppose $t \in \cohen{\lambda}$ and $t \forces$ ``$i \in K$."
Let $t'$ be any extension of $t$.
Because $\dot K \sub N$, it is clear that $t' \forces$ ``$i \in K$" implies $t' \cap N \forces$ ``$i \in K$." 
By our choice of $\mathcal A_i$, this means $t' \cap N$ is compatible with some $s \in \mathcal A_i$; but $s \in N$, so $t'$ is also compatible with $s$.
Hence $t' \not\forces$ ``$i \notin J$."
Because this is true for every $t' \supseteq t$, this shows $t \forces$ ``$i \in J$."
Hence any condition forcing $i \in K$ also forces $i \in J$. It follows that $\0 \forces$ ``$K \sub J$" as claimed.
\end{proof}


If $t \supseteq r$ and, for some $i \in I \cap N$, $t \forces$ ``$i \in K_\g$ for infinitely many $\g \in T$", then $t \cap N \forces$ ``$i \in K_\g$ for infinitely many $\g \in T$."
To see this, note first that $t \forces$ ``$i \in K_\g$ for infinitely many $\g \in T$" 
just means that for any $t' \supseteq t$, there are infinitely many values of $\a$ such that there is some $t_\a$ compatible with $t'$ and $(t_\a,i) \in \dot K_{\g_\a}$. 
But because $\dot K_{\g_\a} \sub N$ for every $\a$ (which means that the $t_\a$'s in the previous sentence are always in $N$), this fact evidently does not change when we replace $t$ with $t \cap N$.

\begin{claim}
For each $\a < \k$, $q_\a \cup r \forces$ ``$\bigwedge K \,\rel\, \bigwedge K_{\g_\a}$."
\end{claim}
\begin{proof}[Proof of Claim]
Fix $\a < \k$.
Let $i,j \in I$, and let $s$ be any extension of $q_\a \cup r$ such that $s \forces$ ``$i \in K_{\g_\a}$ and $j \not\rel i$."
By a previous claim, $q_\a \cup s = s \forces$ ``for any $j' \in I$ with $j' \not\rel i$, there is some $i' \in I$ such that $j' \not\rel i'$ and $i' \in K_\g$ for $\k$-many $\g \in T$."
In particular, $s \forces$ ``there is some $i' \in I$ such that $j \not\rel i'$ and $i' \in K_\g$ for $\k$-many $\g \in T$."

Let $s'$ be any extension of $s$.
There is some $t \supseteq s'$ that decides the value of $i'$ in the previous paragraph: i.e., there is some particular $i' \in I$ such that 
$t \forces$ ``$i' \in K_\g$ for $\k$-many $\g \in T$."
Thus $i' \in I_\k$, and $t \forces$ ``$i' \in K_\g$ for infinitely many $\g \in T$." 
By the paragraph preceding this claim, $t \cap N \forces$ ``$i' \in K_\g$ for infinitely many $\g \in T$." 
Hence $(t \cap N,i') \in \dot K$.
In particular, $t \forces$ ``$i' \in K$."
But also $t \forces$ ``$j \not\rel i'$", and so $t \forces$ ``$j \not\rel \bigwedge K$."
Thus for any $s' \supseteq s$, some extension of $s'$ forces ``$j \not\rel \bigwedge K$."
It follows that $s \forces$ ``$j \not\rel \bigwedge K$."

But $s$ was an arbitrary extension of $q_\a \cup r$ having the property that, for some $i,j \in I$, $s \forces$ ``$i \in K_{\g_\a}$ and $j \not\rel i$."
Therefore $q_\a \cup r \forces$ ``if $i,j \in I$ and $i \in K_{\g_\a}$ and $j \not\rel i$, then $j \not\rel \bigwedge K$."
This implies $q_\a \cup r \forces$ ``$\bigwedge K \rel \bigwedge K_{\g_\a}$."
\end{proof}

In a generic extension $V[H]$ with $r \in H$, we have $\g \in T$ if and only if $q_\a \in H$ for some $\a < \k$ with $\g_\a = \g$, in which case $\dot J_\g = \dot K_{\g_\a}$ and (by the previous claim) $\bigwedge K \,\rel\, \bigwedge K_{\g_\a}$. Therefore
\begin{align}\tag{*}
r \forces \quad & \textstyle \bigwedge K \,\rel\, \bigwedge J_\g \text{ for all } \g \in T.
\end{align}

\begin{claim}
$r \forces$ ``$\dlt \,\rel\, \bigwedge K$."
\end{claim}
\begin{proof}[Proof of Claim]
We will prove separately that $r \forces$ ``$\dlt \rel \bigwedge \set{\bigwedge K_\g}{\g \in T}$" and that $r \forces$ ``$\bigwedge \set{\bigwedge K_\g}{\g \in T} \rel \bigwedge K$."

For the first of these assertions, note that $p \forces$ ``$\dlt \rel \bigwedge (\dot J_\g)_G$ for all $\g \in S$", that $r \supseteq p$, and that $r \forces$ ``$\dot J_\g = \dot K_\g$ for all $\g \in T$ and $T \sub S$." It follows that $r \forces$ ``$\dlt \rel \bigwedge K_\g$ for all $\g \in T$", and therefore $r \forces$ ``$\dlt \rel \bigwedge \set{\bigwedge K_\g}{\g \in T}$."

For the second assertion, first note that, by the definition of $\dot K$, if $i \in I$ then $r \forces$ ``if $i \in K$ then $i \in K_\g$ for infinitely many $\g \in T$."
Hence for every $i \in I$, $r \forces$ ``if $i \in K$ then $\bigwedge \set{K_\g}{\g \in T} \ \rel \ i$"; so $r \forces$ ``for all $i \in I$, if $i \in K$ then $\bigwedge \set{K_\g}{\g \in T} \ \rel \ i$."
Hence $r \forces$ ``$\bigwedge \set{\bigwedge K_\g}{\g \in T} \,\rel\, \bigwedge K$."
\end{proof}

From the last few claims, we see that there is a nice name $\dot J \in N$ for a subset of $I$ such that 
\begin{align*}
r \forces \quad & J = K \text{ and } \textstyle 0 \neq \dlt \,\rel\, \bigwedge K \rel \bigwedge J_\g \rel \g \text{ for all } \g \in T.
\end{align*}
So $r \forces$ ``if $\g \in T$, then $\dot J$ satisfies all the criteria in the definition of $\dot J_\g$." Consequently, $r \forces$ ``$\dot J_\g \sqsubseteq \dot J$ for all $\g \in T$." However, $\card{\set{x \in \mathrm{Lev}(\dot J)}{x \sqsubseteq \dot J}}<\k$, and $\dot J_\g \sqsubseteq \dot J$ implies $\mathrm{Lev}(\dot J_\g) \leq \mathrm{Lev}(\dot J)$. Therefore
$$r \forces \quad \mathrm{Lev}(\dot J_\g) < \mathrm{Lev}(\dot J) \text{ for all but} <\!\k\text{-many } \g \in T.$$
Also $r \forces$ ``$T \sub S \text{ and } |T| = \k$" and therefore
$$r \forces \quad \mathrm{Lev}(\dot J_\g) < \mathrm{Lev}(\dot J) \text{ for } \k\text{-many }\g \in S.$$ 
But $\dot J \in N \sub \bigcup_{\xi < \z}M_\xi$, which implies that $\mathrm{Lev}(\dot J) < \z$.
This contradicts our choice of $\z$ and $p$, because $p$ forces the minimality of $\z$, and $r \supseteq p$.
\end{proof}

\begin{corollary}\label{cor:main1}
$\axiom+\neg\ch$ is consistent relative to \zfc.
\end{corollary}

The proof of Theorem~\ref{thm:cohen} uses a hypothesis stronger than \axiom in $V$ in order to show that \axiom holds in $V[G]$. This leaves open the question of whether such a strong hypothesis in the ground model is really necessary.
\begin{question}
Is \axiom preserved by Cohen forcing?
\end{question}

\section{\gch does not imply \axiom}

In this section we show that \gch does not imply \axiom. 
As mentioned in the introduction, large cardinals are necessary for constructing a model of $\gch+\neg\axiom$.
Another feature of our proof is that the poset $\PP$ for which we show $\axiom(\PP)$ fails has size $\aleph_{\w+1}$. 
This feature is also necessary, in the sense that no smaller poset can work in the presence of \gch.
While in certain models there are smaller posets where \axiom fails (\axiom can fail for a size-$\aleph_2$ poset \cite[Theorem 4.1]{BDMY}, although \axiom always holds for posets of size $\leq\!\aleph_1$ \cite[remark 2.9]{BDMY}), \gch implies that \axiom holds for all posets of size $\leq\!\aleph_\w$.

Consider the following statement:
\begin{itemize}
\item[$\ $] For every model $M$ for a countable language $\mathcal L$ that contains a unary predicate $A$, if $|M| = \k^+$ and $|A| = \k$ then there is an elementary submodel $M' \prec M$ such that $|M'| = \mu^+$ and $|M' \cap A| = \mu$.
\end{itemize}
This statement, abbreviated by writing $(\k^+,\k) \hspace{-.5mm} \twoheadrightarrow \hspace{-.5mm} (\mu^+,\mu)$, is an instance of \emph{Chang's conjecture}. In this section we will consider the case $\k = \aleph_\w$, $\mu = \aleph_0$. This particular instance of Chang's conjecture is known as \emph{Chang's conjecture for $\aleph_\w$} and is abbreviated \chang.

The usual Chang conjecture, which is the assertion $(\aleph_2,\aleph_1) \hspace{-.5mm} \twoheadrightarrow \hspace{-.5mm} (\aleph_1,\aleph_0)$, is equiconsistent with the existence of an $\w_1$-Erd\H{o}s cardinal. Chang's conjecture for $\aleph_\w$ requires even larger cardinals. \chang was first proved consistent relative to a hypothesis a little weaker than the existence of a $2$-huge cardinal in \cite{LMS}. Recently this was improved to a huge cardinal in \cite{EH}. The precise consistency strength of \chang is an open problem, but significant large cardinal strength is known to be needed. This is because \chang implies the failure of $\square_{\aleph_\w}$ (see \cite{SV}, in particular Fact 4.2 and the remarks after it), and the failure of $\square_{\aleph_\w}$ carries significant consistency strength (see \cite{CF}).

\begin{theorem}\label{thm:main2}
If \chang holds, then \axiom fails.
\end{theorem}
\begin{proof}
We will describe a separative ccc poset $\PP$, and then use the Chang conjecture \chang to prove that this poset violates \axiom. The members of $\PP$ have the form $(p,f,A)$, where 
\begin{itemize}
\item[$\circ$] $p \in \HH^{\aleph_\w}$, where $\HH^{\aleph_\w}$ denotes the finite-support product of $\aleph_\omega$ Hechler forcings. The product is indexed by the ordinal $\w_\w$.
\item[$\circ$] $f$ is a function $\w \to \w$, but not the constant function $n \mapsto 0$.
\item[$\circ$] $A$ is a countably infinite subset of $\w_\w$ and $A \supseteq \mathrm{supp}(p)$.
\end{itemize}
Given $(q,g,B),(p,f,A) \in \PP$, we say that $(q,g,B)$ extends $(p,f,A)$ whenever
\begin{itemize}
\item[$\circ$] $q$ extends $p$ in $\HH^{\aleph_\w}$,
\item[$\circ$] $g(n) \geq f(n)$ for all $n \in \w$,
\item[$\circ$] $B \supseteq A$,
\item[$\circ$] if $\a \in A \cap (\mathrm{supp}(q) \setminus \mathrm{supp}(p))$, then $q(\a)$ extends $\< \0,f \>$ in $\HH$.
\end{itemize}
Alternatively, one may think of $\PP$ as a sub-poset of the countable support product of $\aleph_\w$ Hechler forcings, consisting of those conditions $r$ with infinite support such that for all but finitely many coordinates of $\mathrm{supp}(r)$, the $r(\a)$'s are all required to have an empty working part and the same side condition. Under this interpretation, a condition $(p,f,A) \in \PP$ corresponds to the condition $r$ in $\HH^{\aleph_\w}_{\mathrm{ctbl}}$ having countable support $A$, and with $r(\a) = \< \0,f \>$ for all $\a \in A \setminus \mathrm{supp}(p)$.

We begin by verifying that $\PP$ is a separative ccc poset.

\begin{claim}
$\PP$ is separative. 
\end{claim}
\begin{proof}[Proof of claim]
Let $(q,g,B),(p,f,A) \in \PP$ and suppose that $(q,g,B)$ is not an extension of $(p,f,A)$. As there are four parts to the definition of ``extension" in $\PP$, this can mean one of four things.

If $q$ does not extend $p$ in $\HH^{\aleph_\w}$, then because $\HH^{\aleph_\w}$ is separative, there is some $r \in \HH^{\aleph_\w}$ that extends $q$ but is incompatible with $p$. By extending $r$ further if necessary, we may assume $r(\a)$ extends $\<\0,g\>$ for all $\a \in \mathrm{supp}(r)$, and thereby ensure that $(r,g,B)$ is an extension of $(q,g,B)$. Clearly $(r,g,B)$ is incompatible with $(p,f,A)$, because $r$ is incompatible with $p$. 

If $B \not\supseteq A$, then let $\a \in A \setminus B$. Let $h$ be any condition in $\HH$ incompatible with $\<\0,f\>$. (Note that some such $h$ exists because $f$ is not the constant function $n \mapsto 0$.) Let $q' = q \cup \{(\a,h)\}$ and $B' = B \cup \{\a\}$. Then $(q',g,B')$ is a condition extending $(q,g,B)$; but our choice of $\a$ and $h$ guarantees that $(q',g,B')$ is incompatible with $(p,f,A)$. 

If $B \supseteq A$ but $g(n) < f(n)$ for some $n \in \w$, then let $\a \in A \setminus (\mathrm{supp}(p) \cup \mathrm{supp}(q))$ and let $h$ be any condition in $\HH$ extending $\<\0,g\>$ but incompatible with $\<\0,f\>$ (e.g., $h = \<g \rest (n+1),g\>$). Let $q' = q \cup \{(\a,h)\}$. Then $(q',g,B)$ is a condition extending $(q,g,B)$, but it is incompatible with $(p,f,A)$. 

Finally, suppose there is some $\a \in A \cap (\mathrm{supp}(q) \setminus \mathrm{supp}(p))$ such that $q(\a)$ does not extend $\< \0,f \>$ in $\HH$. Then, because $\HH$ is separative, there is some $r \in \HH$ that extends $q(\a)$ but is incompatible with $\<\0,f\>$. Define $q' \in \HH^{\aleph_\w}$ to be identical to $q$, except that $q'(\a) = r$. Then $(q',g,B)$ extends $(q,g,B)$ and is incompatible with $(p,f,A)$. 
\end{proof}

\begin{claim}
$\PP$ has the ccc. 
\end{claim}
\begin{proof}[Proof of claim]
Suppose $\A$ is an uncountable collection of conditions in $\PP$. Let $\B = \set{p}{(p,f,A) \in \A \text{ for some } f \text{ and } A}$ denote the corresponding collection of conditions in $\HH^{\aleph_\w}$. Because $\HH^{\aleph_\w}$ has the ccc, some two conditions in $\B$ are compatible in $\HH^{\aleph_\w}$. But then the two corresponding conditions in $\A$ are also compatible: for if $(q,g,B),(p,f,A) \in \PP$ and $r$ is a common extension of $p$ and $q$ in $\HH^{\aleph_\w}$, then we may further extend $r$, if necessary, so that for each $\a \in \mathrm{supp}(r)$, $r(\a)$ is an extension of both $\<\0,f\>$ and $\<\0,g\>$. Then $(r,\max\{f,g\},A \cup B)$ is a common extension of $(q,g,B)$ and $(p,f,A)$ in $\PP$.
\end{proof}

It remains to show that \chang implies that for any dense $\DD \sub \PP$, there is some condition in $\PP$ that extends uncountably many members of $\DD$. 
Let $\DD$ be a dense sub-poset of $\PP$.

To begin, note that for each countable $A \sub \w_\w$, some member of $\DD$ extends a condition of the form $(p,f,A)$. This implies that \
$$\set{B \sub \w_\w}{(q,g,B) \in \DD \text{ for some } q \in \HH^{\aleph_\w} \text{ and } g \in \w^\w}$$ is cofinal in the poset $([\w_\w]^{\w},\sub)$. The cofinality of this poset is well-known to be $>\! \aleph_\w$. Hence $|\DD| \geq \aleph_{\w+1}$. 

Let $H = \w_\w \cup (\HH^{\aleph_\w} \times \w^\w)$, and note that $|H| = \aleph_\w$.

Let $(M,\in)$ be a model of (a sufficiently large fragment of) \zfc such that $H \sub M$, $\DD \in M$, and $|M| = |M \cap \DD| = \aleph_{\w+1}$. (Such a model can be obtained in the usual way, via the downward L\"{o}wenheim-Skolem Theorem.) Let $\phi: M \to M \cap \DD$ be a bijection, and consider the model $(M,\in,\phi,H)$ for the $3$-symbol language consisting of a binary relation, a unary function, and a unary predicate. Applying the Chang conjecture \chang, there exists some $M' \sub M$ such that $|M'| = \aleph_1$, $H' = M' \cap H$ is countable, and $(M',\in,\phi,H') \prec (M,\in,\phi,H)$.

Let $\DD' = \DD \cap M'$. By elementarity, the restriction of $\phi$ to $M'$ is a bijection $M' \to \DD'$, and so $|\DD'| = \aleph_1$.

Let $B = \w_\w \cap M'$. As $B \sub H'$, we have $|B| = \aleph_0$. Note that $(p,f,A) \in \DD'$ implies $A \in M'$, and therefore (because $A$ is countable, and $M'$ models (enough of) \zfc) $A \sub M'$. Therefore $(p,f,A) \in \DD'$ implies $A \sub B$. 

Furthermore, $(p,f,A) \in \DD'$ implies $(p,f) \in M'$, which implies $(p,f) \in H'$.
Therefore 
$$\set{(p,f)}{ (p,f,A) \in \DD' \text{ for some } A \sub B}$$
is countable. 
But $\DD'$ is uncountable, so by the pigeonhole principle, there is some pair $(p,f) \in \HH^{\aleph_\w} \times \w^\w$ such that $\set{ A \sub B }{ (p,f,A) \in \DD' }$ is uncountable.

Finally, note that $(p,f,B)$ is a condition in $\PP$, and that $(p,f,B)$ extends $(p,f,A)$ whenever $A \sub B$. Therefore $(p,f,B)$ extends uncountably many conditions in $\DD$. 
\end{proof}

\begin{corollary}
$\gch+\neg\axiom$ is consistent relative to a huge cardinal.
\end{corollary}


\section{The measure algebra of weight $\aleph_\omega$}

In \cite[Section 4]{BDMY}, it is observed that \ma implies \axiom fails for the weight-$\aleph_0$ measure algebra. In fact, this was the first known example of a poset for which \axiom consistently fails. The results contained in this section and the previous one grew from trying to discover whether \chang implies \axiom fails for the weight-$\aleph_\w$ measure algebra. As mentioned in the previous section, \chang does not imply the failure of \axiom for any poset of size $\leq\!\aleph_\w$, so this makes the weight-$\aleph_\w$ measure algebra a natural place to look. We still do not know whether \chang implies the failure of \axiom for the weight-$\aleph_\w$ measure algebra. But we show below that $\gch\,+\,$\chang is consistent with the failure of \axiom for the weight-$\aleph_\w$ measure algebra.

Given some set $A$, $2^A$ denotes the set of all functions $A \to 2$. 
The product measure $\mu$ on $2^A$ is defined by setting 
$$\mu\!\left(\set{f \in 2^A}{f(\a) = 0}\right) = \mu\!\left(\set{f \in 2^A}{f(\a) = 1}\right) = \textstyle \frac{1}{2}$$
for all $\a \in A$. 
More precisely, this coordinate-wise assignment extends naturally to a pre-measure on the clopen subsets of $2^A$, and this extends, via Carath\'eodory's Theorem, to a countably additive measure on the smallest $\s$-algebra containing all the clopen subsets of $2^A$. We denote this $\s$-algebra by $\B_A$.

Now suppose $A = \k$ is an infinite cardinal number, and let $M_\k$ denote the quotient of $\B_\k$ by the ideal of sets having $\mu$-measure $0$. Then $M_\k$ is a $\s$-complete Boolean algebra, called \emph{the measure algebra of weight} $\k$.

Given $X \sub 2^\k$ and $A \sub \k$, we say that $X$ is \emph{supported} on $A$ if there is some $Y \sub 2^A$ such that $X = Y \times 2^{\k \setminus A}$.
It is easy to check that if $X \neq \0$ and $X$ is supported on every $A$ in some collection $\A \sub \mathcal P(\k)$, then $X$ is supported on $\bigcap \A$. Therefore there is a smallest $A \sub \k$ on which $X$ is supported, and we denote this set by $\mathrm{supp}(X)$.

\begin{lemma}\label{lem:essentiallycountable}
Every member of $\B_\k$ is supported on a countable subset of $\k$. In fact, $X \in \B_\k$ if and only if $X = Y \times 2^{\k \setminus A}$ for some countable $A \sub \k$ and some Borel $Y \sub 2^A$.
\end{lemma}
\begin{proof}
Let $\B$ denote the set of all $X$ such that $X = Y \times 2^{\k \setminus A}$ for some countable $A \sub \k$ and some Borel $Y \sub 2^A$. It is clear that $\B$ is a $\s$-algebra containing all the basic clopen subsets of $2^\k$; hence $\B_\k \sub \B$. Conversely, if $A \sub \k$ is countable, then $\B_\k$ contains $C \times 2^{\k \setminus A}$ for every clopen $C \sub 2^A$, because $C \times 2^{\k \setminus A}$ is clopen in $2^\k$. It follows that $\B_\k$ must contain $Y \times 2^{\k \setminus A}$ for every Borel $Y \sub 2^A$. Hence $\B \sub \B_\k$.
\end{proof}

Let $\AA$ denote the amoeba forcing. Conditions in $\AA$ are open subsets of $2^\w$ with measure $<\!\frac{1}{2}$, and the extension relation on $\AA$ is $\sub$. Let $\AA^\w$ denote the finite support product of $\w$ copies of $\AA$.

\begin{lemma}\label{lem:amoeba}
Let $V$ be a model of \zfc and let $G$ be an $\AA^\w$-generic filter over $V$. In $V[G]$, there is a countable collection $\C$ of non-null closed subsets of $2^\w$ such that if $B$ is any non-null Borel subset of $2^\w$ whose Borel code is in $V$, then there is some $C \in \C$ such that $C \sub B$.
\end{lemma}
\begin{proof}
Each $p \in G$ is a sequence of open subsets of $2^\w$ in $V$, all but finitely many of which are $\0$. For each $p \in G$ and $n \in \w$, let $\widetilde{p}(n)$ denote the reinterpretation of $p(n)$ in $V[G]$; i.e., $\widetilde{p}(n)$ is the $V[G]$-interpretation of the Borel code of $p(n)$ in $V$.
For each $n \in \w$, define $U_n = \bigcup_{p \in G}\widetilde{p}(n)$, and let
$\C = \set{2^\w \setminus (U_n \cup U_m)}{m,n \in \w}.$

It is straightforward to show that each $U_n$ is an open set with measure $\frac{1}{2}$. Fix $m,n \in \w$. The set of all $p \in \AA^\w$ with $p(m) \cap p(n) \neq \0$ is dense. Therefore $U_m \cap U_n \neq \0$, and because both these sets are open, $\mu(U_m \cap U_n) > 0$. Hence 
\begin{align*}
\mu(2^\w \setminus (U_m \cup U_n)) & = 1 - \mu(U_m \cup U_n) \\
& = 1 - (\mu(U_m) + \mu(U_n) - \mu(U_m \cap U_n)) \\
& = \mu(U_m \cap U_n) > 0.
\end{align*}
Thus $\C$ is a countable collection of non-null closed subsets of $2^\w$.

Let $B$ be a non-null Borel set in $V$. Then $\mu(2^\w \setminus B) < 1$, and this implies there is an open $W \sub 2^\w$ such that $\mu(W) < 1$ and $2^\w \setminus B \sub W$. Any open set of measure $<\!1$ can be split into two open sets of measure $<\!\frac{1}{2}$, so in particular there are open $V_1,V_2 \sub 2^\w$ such that $\mu(V_1) < \frac{1}{2}$, $\mu(V_2) < \frac{1}{2}$, and $V_1 \cup V_2 = W$. 
Now the set of all $p \in \AA^\w$ with $p(m) = V_1$ and $p_n = V_2$ for some $m,n \in \w$ is dense. Therefore there exist some $m,n \in \w$ and some $p \in G$ such that $p(m) = V_1$ and $p(n) = V_2$. 
Letting $\widetilde B$, $\widetilde{V}_1$, and $\widetilde{V}_2$ denote the $V[G]$-interpretations of the Borel codes for $B$, $V_1$, and $V_2$, respectively, we have $2^\w \setminus \widetilde B \sub \widetilde{V}_1 \cup \widetilde{V}_2 \sub \widetilde{p}(m) \cup \widetilde{p}(n)$. Hence $\widetilde B \supseteq 2^\w \setminus (U_n \cup U_m) \in \C$.
\end{proof}

\begin{theorem}
It is consistent, relative to a huge cardinal, that \gch holds and that \axiom fails for $M_{\aleph_\w}$.
\end{theorem}
\begin{proof}
Let $V$ be a model of \gch plus \chang. Recall that the existence of such a model is consistent relative to a huge cardinal.

Let $\AA$ denote the amoeba forcing, and let $\PP$ denote the length-$\w_1$, finite support iteration of $\AA^\w$. Let $G$ be a $V$-generic filter on $\PP$. We claim that $V[G]$ is the desired model of \gch where $\axiom(M_{\aleph_\w})$ fails.

A standard argument shows $V[G] \models \gch$. 
Therefore, to prove the theorem we must show that $\axiom(M_{\aleph_\w})$ fails in $V[G]$. Because $M_{\aleph_\w}$ has the ccc, this amounts to showing that for any dense sub-poset $\DD$ of $M_{\aleph_\w} \setminus \{\mathbf 0\}$, some member of $M_{\aleph_\w} \setminus \{\mathbf 0\}$ has uncountably many members of $\DD$ above it.

We obserbve that $\PP$ has the ccc, and it is know that \chang is preserved by ccc forcing. (This fact is considered folklore, but a proof can be found in \cite[Lemma 13]{EH}.) Hence $V[G] \models$ \chang.

It follows from Lemma~\ref{lem:essentiallycountable} that every $X \in \B_{\w_\w}$ can be represented in a canonical fashion by a pair $(A,a)$, where $A = \mathrm{supp}(X)$ is countable, and $a$ is some canonical code for the Borel subset $Y$ of $2^A$ such that $X = Y \times 2^{\w_\w \setminus A}$.
Let us call the pair $(A,a)$ the \emph{code} for $X$.

For each $\a < \w_1$, let $G_\a$ denote (as usual) the restriction of $G$ to the first $\a$ coordinates of $\PP$. 
For each $\a < \w_1$, let $B^\a_{\w_\w}$ denote the set of all those members of $B_{\w_\w}$ whose code is in $V[G_\a]$.
For every $X \in \B_{\w_\w}$, the code for $X$ consists of a countable set of ordinals and a countable sequence of integers. This implies there is some $\a < \w_1$ such that the code for $X$ is a member of $V[G_\a]$.
Hence $B_{\w_\w} = \bigcup_{\a < \w_1}B^\a_{\w_\w}$.

Working in $V[G]$, let $\DD$ be a dense sub-poset of $M_{\aleph_\w}$. In what follows, it is easier to work with members of $\B_{\w_\w}$ rather than with their equivalence classes in $M_{\aleph_\w}$. For each $Z \in \DD$, fix some $X_Z \in \B_{\w_\w}$ representing $Z$. Let $\EE = \set{X_Z}{Z \in \DD}$, and observe that $|\EE| = |\DD|$. Because every dense sub-poset of $M_{\aleph_\w}$ has cardinality $>\!\aleph_\w$ (see e.g. \cite[Theorem 6.13]{Fremlin}), $|\EE| > \aleph_\w$. Also $|\EE| \leq 2^{\aleph_\w} = \aleph_{\w+1}$, and therefore $|\EE| = \aleph_{\w+1}$.

Because $B_{\w_\w} = \bigcup_{\a < \w_1}B^\a_{\w_\w}$ and $\card{\EE} = \aleph_{\w+1}$, there is some $\a < \w_1$ such that $\card{\EE \cap B^\a_{\w_\w}} = \aleph_{\w+1}$. Fix some such $\a$, and let $\EE_\a = \EE \cap B^\a_{\w_\w}$.

Let $(M,\in)$ be a model of (a sufficiently large fragment of) \zfc such that $\w_\w \sub M$, $\EE_\a \in M$, and $|M| = |M \cap \EE_\a| = \aleph_{\w+1}$. (Such a model can be obtained in the usual way, via the downward L\"{o}wenheim-Skolem Theorem.) Let $\phi: M \to M \cap \EE_\a$ be a bijection, and consider the model $(M,\in,\phi,\w_\w)$ for the $3$-symbol language consisting of a binary relation, a unary function, and a unary predicate. Applying the Chang conjecture \chang, there exists some $M' \sub M$ such that $|M'| = \aleph_1$, $M' \cap \w_\w$ is countable, and $(M',\in,\phi,\w_\w) \prec (M,\in,\phi,\w_\w)$.

Let $\EE_\a' = \EE_\a \cap M'$. By elementarity, the restriction of $\phi$ to $M'$ is a bijection $M' \to \EE_\a'$, and so $|\EE_\a'| = \aleph_1$.

Let $A = \w_\w \cap M'$. If $X \in \EE_\a'$, then $\mathrm{supp}(X) \in M'$, and therefore (because $\mathrm{supp}(X)$ is countable, and $M'$ models (enough of) \zfc) $\mathrm{supp}(X) \sub M'$. Hence $X \in \EE_\a'$ implies $\mathrm{supp}(X) \sub A$.

By Lemma~\ref{lem:amoeba}, in $V[G]$ there is a countable collection $\C$ of non-null closed subsets of $2^A$ such that if $B$ is any non-null Borel subset of $2^A$ whose Borel code is in $V[G_\a]$, then there is some $C \in \C$ such that $C \sub B$. (Strictly speaking, our lemma gives us such a family in $V[G_{\a+1}]$. But by reinterpreting the Borel codes of the members of that family in $V[G]$, we obtained the desired collection $\C$.) In particular, every $X \in \EE_\a'$ contains $C \times 2^{\w_\w \setminus A}$ for some $C \in \C$. By the pigeonhole principle, there is some particular $C \in \C$ such that $X \supseteq C \times 2^{\w_\w \setminus A}$ for uncountably many $X \in \EE_\a'$.

Moving from representatives back to equivalence classes, $[C \times 2^{\w_\w \setminus A}] \neq [\0]$ because $C$ is non-null in $2^A$, and $[C \times 2^{\w_\w \setminus A}] \leq [X]$ for uncountably many $X \in \EE_\a'$. Hence $[C \times 2^{\w_\w \setminus A}] \in M_{\aleph_\w} \setminus \{\mathbf 0\}$ and $[C \times 2^{\w_\w \setminus A}]$ extends uncountably many members of $\DD$. Because $\DD$ was an arbitrary dense sub-poset of $M_{\aleph_\w}$, and because $M_{\aleph_\w}$ has the ccc, we conclude that $\axiom(M_{\aleph_\w})$ fails.
\end{proof}


\end{document}